\documentclass[12pt]{amsart}
\usepackage[below]{placeins}

\usepackage{amsmath,amscd,amssymb, epsfig,psfrag,subfigure}

\newtheorem{theorem}{Theorem}[section]
\newtheorem{lemma}[theorem]{Lemma}

\newtheorem{prop}[theorem]{Proposition}
\newtheorem{cor}[theorem]{Corollary}
\theoremstyle{definition}
\newtheorem{definition}[theorem]{Definition}

\newtheorem{rmk}[theorem]{Remark}

\def\L{\mathcal{L} }
\def\T{\mathcal{T}}

\def\S{\Sigma}

\def\HH{\mathbb{H}}

\def\HH{\mathbb{H}}

\def\split{\rightharpoonup}

\edef\t@mp{\catcode`\noexpand\~=\the\catcode`\~}%
    \catcode`\~=12
    \def\tild@{~}%
    \t@mp

\begin{document}

\title{Ideal triangulations of pseudo-Anosov mapping tori} 

\author[Ian Agol]{%
        Ian Agol} 
\address{%
    University of California, Berkeley \\
    970 Evans Hall \#3840 \\
    Berkeley, CA 94720-3840} 
\email{%
     ianagol@math.berkeley.edu}  
\thanks{Agol partially supported by NSF grant DMS-0806027}
\subjclass[2000]{57M}

\date{%
 Januar, 2010}

\dedicatory{Dedicated to Bus Jaco on the occasion of his 70th birthday}


\begin{abstract} 
We show how to construct an ideal triangulation of a mapping
torus of a pseudo-Anosov map punctured along the singular fibers. This gives rise to 
a new conjugacy invariant of mapping classes, and a new proof of a theorem of Farb-Leininger-Margalit. 
The approach in this paper is based on ideas of Hamenstadt. 
\end{abstract} 

\maketitle
\section{Introduction}
Recently, Farb, Leininger and Margalit \cite{FLM} proved that the mapping tori of small dilatation
mapping classes of closed surfaces are obtained by Dehn filling on finitely many cusped hyperbolic
3-manifolds. Recent results of Hamenstadt \cite{Hamenstadt09} suggested another approach to this
theorem using splitting sequences of train tracks. In discussions with Hamenstadt, she suggested that her complex
of train track splitting sequences should give a new solution to the conjugacy problem for pseudo-Anosov mapping
classes, somewhat similar to work of Mosher on ``circular expansion complexes" \cite{Mosher83, Mosher86, Mosher03}. 
In this paper we give an exposition of these results from the perspective
of measured train track splitting sequences. Instead of a complex of splittings, we obtain a taut ideal triangulation
associated to a mapping class. The fact that the reverse of splitting sequences (folding sequences) give
rise to Perron-Frobenius maps of the weight spaces of train tracks enables us to give a new proof of the main result of \cite{FLM}. 

What Hamenstadt proved is that if one is given a minimal lamination $\L$ on a surface, there is a complex
of splitting sequences of train tracks carrrying $\L$ which forms a CAT(0) cube complex in a natural way, with
vertices corresponding to train tracks carrying the lamination $\L$, directed edges corresponding to $\L$-splits of the train tracks, and cubes corresponding
to commuting collections of $\L$-splits \cite{Hamenstadt09}. This is analogous to the complex constructed by Mosher \cite{Mosher03}, however
we remark that the definition of ``splitting sequence" in that manuscript is distinct from the usage in this paper (Mosher allows shifts/slides as well).
If one takes the stable lamination associated to a pseudo-Anosov map $\varphi$, then one may find a bi-infinite cube
complex of train tracks which is invariant under $\varphi$, and this gives rise to an invariant of the conjugacy
class of $\varphi$. Instead of considering all sequences of splittings in this paper, we consider ``maximal splitting"
sequences, where the train tracks have a measure, and the splittings occur at the branches of the train track of maximal weight.
We prove that these sequences are eventually periodic for a pseudo-Ansov stable lamination in Section \ref{splitting sequences} (this is
somewhat analogous to the approach of \cite{PapadopoulosPenner87, PapadopoulosPenner90} and \cite[Lemma 10.2.6]{Mosher03}). 
These sequences give something like a continued fraction expansion for pA mapping classes, analogous to the
case of Anosov maps of the 2-torus (Mosher observed the analogy between continued fractions and his
circular expansion complexes \cite{Mosher03}). This approach allows us to reprove the main result of \cite{FLM} in Section \ref{dilatation}. 
We also obtain a layered ideal  triangulation of the mapping torus punctured at the singular fibers,
from Whitehead moves on the triangulations dual to the  train tracks in the splitting sequence in Section \ref{triangulations}. 
The natural structure of these triangulations is a taut ideal triangulation, introduced by Lackenby \cite{Lackenby00}.

Acknowledgements: We thank Ursula Hamenstadt, Chris Leininger, and Saul Schleimer for helpful
conversations. We thank Matthias Goerner, Yi Liu, and Lee Mosher for making comments on an earlier draft. 
We thank Marc Lackenby for allowing us to use some of his figures. 

\section{Definitions}
We review some background definitions and establish some notation in this section.

Let $\Sigma=\Sigma_{g,n}$ be an orientable surface of genus $g$ with $n$ punctures, and let $\varphi: \Sigma\to \Sigma$ be a homeomorphism. 
By the Nielsen-Thurston classification of mapping classes, $\varphi$ is conjugate to either an element of finite order, a reducible mapping class,
or to a pseudo-Anosov map \cite{Th88, Poenaru80}. In the third case, there are  the stable and unstable measured laminations $\L^s$ and $\L^u \in \mathcal{ML}(\S)$ associated to $\varphi$ such that 
for any curve $c\subset \Sigma$, $[\varphi^n(c)] \to [\L^s] \in \mathcal{PL}(\S)$, and $[\varphi^{-n}(c)]\to [\L^u]\in \mathcal{PL}(\S)$ (see \cite[Chapter 8]{Th} for the
notation). Moreover,
$\L^s$ and $\L^u$ meet every essential closed curve in $\S$ which is not isotopic into a neighborhood of a puncture. Further, there exists a
dilatation $\lambda(\varphi) \in (1,\infty)$, $\varphi(\L^s)=\lambda(\varphi)\L^s$ and $\varphi(\L^u)=\lambda(\varphi)^{-1} \L^u$, up to isotopy. These 
laminations are only unique up to scaling and isotopy. 
If we choose a complete finite area hyperbolic metric on $\S$, then we may isotope the  measured laminations $\L^s$ and $\L^u$ to have totally
geodesic leaves. Each complementary region will be isometric to an ideal polygon with at least three sides, 
or to a punctured disk with at least one puncture on the boundary. If we puncture $\S$ at a point in each ideal polygon
complementary region, we obtain a surface $\S^{\circ}$, and we have a well-defined (up to isotopy) restriction
map $\varphi^\circ=\varphi_{|\S^\circ} : \S^{\circ}\to \S^\circ$. Moreover, $\lambda(\varphi^\circ)=\lambda(\varphi)$, and we may assume
$[\varphi^\circ(\L^{u,s})]  =[\L^{u,s}] \in \mathcal{PL}(\S^\circ)$. 

To encode a lamination combinatorially, we use train tracks (after Thurston \cite[Ch. 8]{Th}). A train track $\tau\subset \Sigma$ is a 1-complex with
trivalent vertices which is locally modeled on a switch (see Figure \ref{switch}(a)) and satisfying some extra conditions. Edges of the train track are
called {\it branches}, and vertices are called {\it switches}.   Each branch $e\subset \tau$
is $C^1$ embedded, and at each switch of $\tau$, there is a well-defined tangent space to the branches coming into the switch.
A {\it half-branch} is an end of the interior of a branch. Each switch is in the closure of three half-branches. 
 A {\it large half-branch} is the end of a half-branch meeting a switch
on the the side of the tangent space to the switch 
opposite of the other two half-branches incident with the switch, and a {\it small half-branch} is a half-branch meeting
the switch on the same side of the tangent space to the switch as another branch (in Figure \ref{switch}(b), the half-branch
labelled $a$ is large, whereas the half-branches labeled $b$ and $c$ are small). 
For each component $R$ of $\S-\tau$, the boundary $\partial R$ is a piecewise smooth curve. 
A non-smooth point of $\partial R$ is a {\it cusp}. 
For each component $R$ of $\S -\tau$, the double of $R$ along $\partial R$ with the
cusps removed must have negative Euler characteristic.  We will
follow the notation and conventions of Penner-Harer \cite{PennerHarer92}. A {\it large branch} is a branch with 
both ends being large half-branches, and  a {\it small branch} has both ends small half-branches. Otherwise,
a branch is {\it mixed}. 
When all of the complementary regions of the
train track $\tau$ are punctured disks, then it is convenient to label the branches of the track by the isotopy class of the arc type
of an ideal edge dual to the branch. This gives an unambiguous way of labeling branches in differing train tracks on the
same surface. 

\begin{figure}[htb] 
	\begin{center}
	\psfrag{a}{$a$}
	\psfrag{b}{$b$}
	\psfrag{c}{$c$}
	\psfrag{d}{$a=b+c$}
	\subfigure[Train track switch]{\epsfig{figure=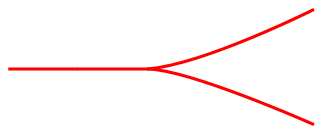,angle=0,width=.45\textwidth}}\quad
	\subfigure[Measured train track switch]{\epsfig{figure=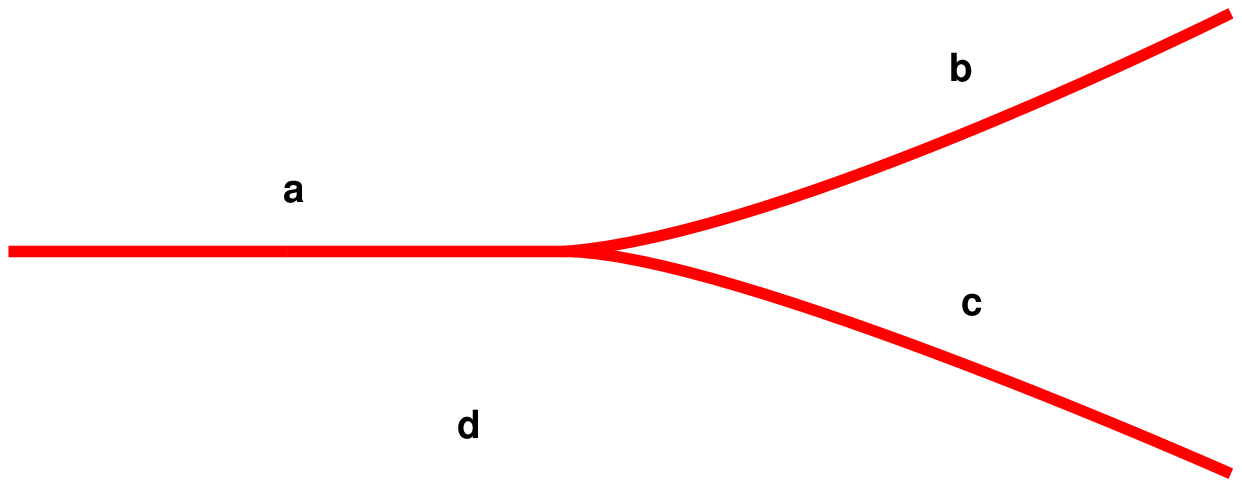,angle=0,width=.45\textwidth}}
	\caption{\label{switch} Train track switches}
	\end{center}
\end{figure} 

A \emph{measured train track} is a train track $\tau$ together with a transverse measure 
$\mu$ which is a function assigning a weight to each edge of $\tau$. We will denote a
measured train track as a pair $(\tau,\mu)$. At each switch of $\tau$ with incoming
edges $b,c$ and outgoing edge $a$, the weights of 
the adjacent edges must satisfy $\mu(a)=\mu(b)+\mu(c)$ (Figure \ref{switch}(b)). 
In some pictures, we will abuse notation and say $a=b+c$ when we really mean $\mu(a)=\mu(b)+\mu(c)$.
This is especially convenient when the train track changes, but most of the labelled edges do not. 
The {\it weight space} $W(\tau)$ is the convex space of all positive measures on $\tau$. 
The lamination $\L$ is {\it suited to } the train track $\tau$ if there is a differentiable
map $f:\S\to\S$ homotopic to the identity such that $f(\L)=\tau$, and $f$ is non-singular
on the tangent spaces to the leaves of $\L$. If $\L$ is a measured lamination and 
$\L$ is suited to $\tau$, then $\L$ induces a measure $\mu$ on $\tau$ given by
the measure on $\tau$ which pulls back to the transverse measure on $\L$ via
the map $f$. 

A {\it splitting} of a measured train track $(\tau,\mu)$ is a move at a large branch which splits it
according to the weights of the neighboring edges (see Figure \ref{measuresplit}). We use
the notation $(\tau,\mu) \split_e (\sigma, \nu)$.  
\begin{figure}[htb] 
	\begin{center}
	\psfrag{a}{$a$}
	\psfrag{c}{$c$}
	\psfrag{b}{$b$}
	\psfrag{d}{$d$}
	\psfrag{e}{$a+b=e=c+d$}
	\psfrag{f}{$e'=c-a$}
	\psfrag{g}{$(\tau,\mu)$}
	\psfrag{h}{$(\sigma,\nu)$}
	\epsfig{file=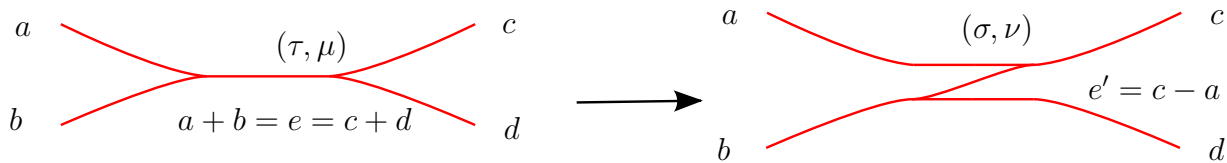, width=\textwidth}
	\caption{\label{measuresplit} Splitting a large branch when $a<c$. }
	\end{center}
\end{figure} 

If the measured lamination $\L$ is suited to $(\tau,\mu)$, then $\L$ will also be
suited to  $(\sigma,\nu)$ when $(\tau,\mu)\split_e (\sigma,\nu)$. If the weights $\mu(a)=\mu(c)$,
then the branch $e'$ will not appear in the train track $\sigma$, but we do not consider this to be
a split in this paper. Conversely, any measured
train track $(\tau,\mu)$ gives rise to a unique measured lamination $\L$ suited to $(\tau,\mu)$.

Two other moves on train tracks do not depend on the measure. A {\it shift} changes
the track at a mixed branch (Figure \ref{shifting}) (these are also called {\it slides} \cite{Mosher03}). A {\it fold} changes the track at a small branch (Figure \ref{fold}). The shift
is self-inverse, whereas the fold is inverse to a split. However, the split depends on the
measure, whereas the fold does not, which is why it is important to make a distinction. 
\begin{figure}[htb] 
	\begin{center}
	\epsfig{file=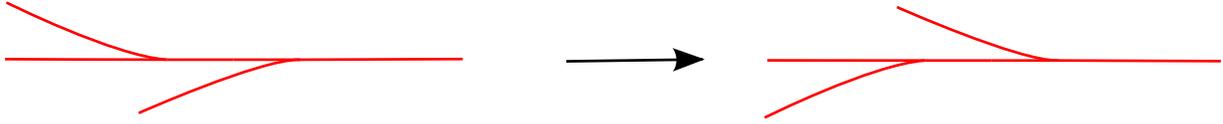, width=\textwidth}
	\caption{\label{shifting} Shifting a mixed branch. }
	\end{center}
\end{figure} 
\begin{figure}[htb] 
	\begin{center}
	\psfrag{a}{$a$}
	\psfrag{c}{$c$}
	\psfrag{b}{$b$}
	\psfrag{d}{$d$}
	\psfrag{e}{$e'=a+e+d$}
	\psfrag{f}{$e$}
	\psfrag{g}{$(\tau,\mu)$}
	\psfrag{h}{$(\sigma,\nu)$}
	\epsfig{file=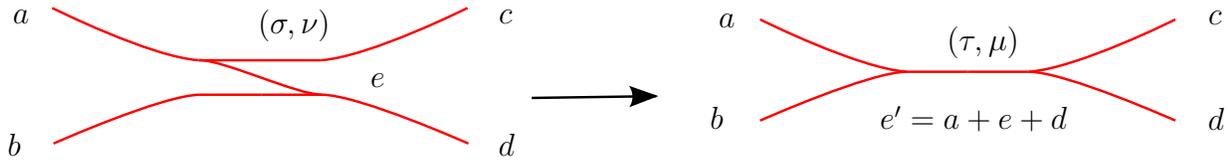, width=\textwidth}
	\caption{\label{fold} Folding a small branch. }
	\end{center}
\end{figure} 

Two measured train tracks $(\tau,\mu),(\tau',\mu')$ are {\it equivalent} if there
is a sequence of train tracks obtained by splits, shifts, folds and isotopies which interpolates
between the two. 

A {\it maximal splitting} of $(\tau, \mu)$ is a splitting along all of the large branches
of $\tau$ with maximal $\mu$ weight. 
We denote $(\tau_0,\mu_0) \split (\tau_1,\mu_1)$ if $(\tau_1,\mu_1)$
is obtained from $(\tau_0,\mu_0)$ by a single maximal splitting. If we have
$n$ maximal splittings $(\tau_0,\mu_0)\split (\tau_1,\mu_1) \cdots \split (\tau_n,\mu_n)$,
then we denote $(\tau_0,\mu_0)\split^n (\tau_n,\mu_n)$, or if we don't want to specify $n$, we
will use $\split^*$. 

A \emph{trainpath} on a train track $\tau$ is a
$C^1$-immersion $\rho:[m,n]\to \tau\subset \S$
which maps each interval
$[k,k+1]$ $(m\leq k\leq n-1)$ onto a branch of $\tau$. The integer
$n-m$ is then called the \emph{length} of $\rho$.
 Each complementary region of $\tau$
is bounded by a finite number of trainpaths
which either are simple closed curves or terminate
at the cusps of the region. A lamination is {\it minimal} if every half-leaf is dense.

\begin{lemma} \label{edgesplit}
Let $\L$ be a minimal lamination, and let $\L$ be suited to $(\tau,\mu)$.  Let $e\subset \tau$ be a
branch of $\tau$, and let $(\tau,\mu) \split (\tau_1,\mu_1) \split (\tau_2,\mu_2) \split\cdots \split (\tau_n,\mu_n) \split \cdots$ be an infinite
 sequence of maximal splittings of $(\tau,\mu)$. Then there exists $n$ such that $(\tau_n,\mu_n) \split  (\tau_{n+1},\mu_{n+1})$
splits the branch $e$ (so that $\mu_n(e)$ is maximal weight for $\tau_n$).
\end{lemma}
\begin{proof}
We may assume that we have a map
$f:\S\to \S$ such that $f(\L)=\tau$ since $\L$ is suited to $\tau$. For any half-leaf  $l\subset \L$, since
$\L$ is minimal, $f(l)$ must eventually cross the branch $e$. For each cusp of $\tau$, there are two 
half-leaves of $\L$ which are parallel for all time (since $\L$ is suited to $\tau$) whose start
is adjacent to the cusp (these leaves are the ideal arcs corresponding to the cusp in the complementary
ideal polygon region of $\L$). For each cusp $c$ of 
$\tau$, let $\rho_c: [n,m]\to \tau$ be a trainpath which is parallel to the path $f(l)$ emanating from
the cusp $c$, such that $\rho_c([m-1,m])=e$. Each time we split $\tau_k \split \tau_{k+1}$ at a branch
adjacent to a cusp $c$, the trainpath $\rho_c$ shrinks to $\rho_c: [n-1,m]\to \tau$. 
Thus, at each stage of splitting, the total length of such cusp paths decreases by $2$ (since
each splitting branch is adjacent to two cusps), and thus we see that eventually we must split the branch
$e$. 
\end{proof}

\section{Splitting sequences} \label{splitting sequences}
We state some basic results about measured train tracks and measured laminations.

\begin{theorem}  \cite[Prop. 8.9.2]{Th}, \cite[Theorem 4.1]{PP87}
If $\L$ is a measured lamination, then $\L$ is suited to a measured train track $(\tau,\mu)$.
\end{theorem}

\begin{theorem} \label{split1} \cite[Theorem 2.8.5]{PennerHarer92}
Measured train tracks  $(\tau_i, \mu_i)$  give rise to the
same measured lamination $\L$ (up to isotopy) if and only 
if they are equivalent. 
\end{theorem}

Together, these theorems imply that  equivalence classes of measured train tracks
are in one-to-one correspondence to measured laminations. Equivalence classes of measured
train tracks are generated by splitting, shifting, and folding. The next theorem implies that shifting
is not needed. 

\begin{theorem} \label{split2} \cite[Theorem 2.3.1]{PennerHarer92}
If $(\tau_i, \mu_i)$ are equivalent positively measured train tracks, then there is
a train track $(\tau, \mu)$ such that $(\tau_i,\mu_i)$ splits (and isotopes) to $(\tau,\mu)$
for $i=1,2$.  
\end{theorem}
This theorem is proven by demonstrating that two train tracks which are equivalent
by a split, shift, or fold, have common splittings (having a common splitting is not 
hard to show is an equivalence relation using the technique of Corollary \ref{spliteq}). 
For splits and folds this is clear. For two train
tracks related by a shift, one performs the same sequence of splits of maximal edges
to a sequence of two train tracks related by shifts 
to see that once the middle incoming edge $b$ becomes maximal, the
two resulting train tracks are the same (see Figure \ref{shift2}). 

\begin{figure}[htb] 
	\begin{center}
	\psfrag{a}{}
	\psfrag{b}{$b$}
	\psfrag{c}{}
	\psfrag{d}{}
	\psfrag{e}{}
	\psfrag{f}{$\downharpoonleft$}
	\psfrag{g}{$\overset{shift}{\rightarrow}$}
	\psfrag{h}{$=$}
	\epsfig{file=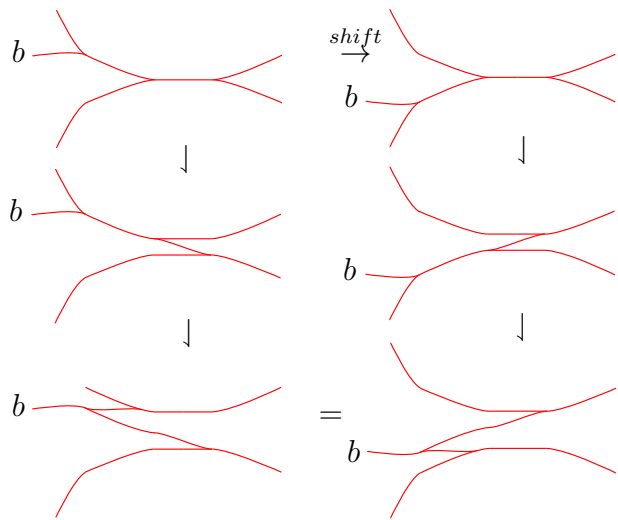, width=.5\textwidth}
	\caption{\label{shift2} The effect of shifting then splitting.}
	\end{center}
\end{figure} 

\begin{cor} \label{spliteq}
Suppose $\L$ is suited to  $(\tau, \mu)$, $(\tau',\mu')$. Then 
there exists $(\tau'',\mu'')$ such that $(\tau,\mu)\split^* (\tau'',\mu'')$ and $(\tau',\mu')\split^* (\tau'',\mu'')$. 
\end{cor}
\begin{proof}
The property of having a common maximal split is  transitive, so
by Theorem \ref{split2}, we need only assume that  $(\tau,\mu)\split_e (\tau',\mu')$ for
some large branch $e\subset \tau$. If $e$ is the sole branch of maximal weight of $(\tau,\mu)$,
then we have $(\tau,\mu)\split (\tau',\mu')=(\tau'',\mu'')$ (a maximal split). If there are other branches of $(\tau,\mu)$
of maximal weight $=\mu(e)$, then we have   $(\tau,\mu)\split (\tau'',\mu'')$
and $(\tau',\mu')\split (\tau'',\mu'')$, where this maximal splitting is along
the edges of $\tau'$ which have the same $\mu'$ weight as $\mu(e)$.

Otherwise, assume that $\mu(e)$ is not a maximal weight of $\mu$. Let $(\tau,\mu)=(\tau_1,\mu_1) \split (\tau_2,\mu_2) \split \cdots $ 
be the maximal splitting sequence. By Lemma \ref{edgesplit}, there exists $n$ such that $\mu_n(e)$ is a maximal weight of $\mu_n$. 

We claim that $(\tau_1,\mu_1) \split^{n} (\tau_{n+1},\mu_{n+1})$ and $(\tau',\mu')\split^{n-1} (\tau_{n+1},\mu_{n+1})$. 
We'll assume that $(\tau'_i,\mu'_i)\split (\tau'_{i+1}, \mu'_{i+1})$, for $i=2,\ldots, n$, where $\tau'_2=\tau'$. 
By induction, we prove that $(\tau_i,\mu_i)\split_e (\tau'_{i+1},\mu'_{i+1})$ for $i<n$. Once $e$ becomes a maximal weight edge of $(\tau_n,\mu_n)$,
there may be other maximal edges of $(\tau_n,\mu_n)$ of the same weight as $\mu_n(e)$, but we will have $(\tau_n,\mu_n)\split (\tau_{n+1},\mu_{n+1})=(\tau'_{n+1},\mu'_{n+1})$. 
$$
\begin{array}{ccccc}
(\tau_1,\mu_1) &\split &(\tau_2,\mu_2) & \split \cdots \split & (\tau_n,\mu_n) \\
\downharpoonleft_e & & \downharpoonleft_e & & \downharpoonleft_e   \\
(\tau'_2,\mu'_2)  & \split & (\tau'_3,\mu'_3) & \split \cdots \split & (\tau_{n+1},\mu_{n+1})
\end{array}
$$

Let $e'$ be the new edge of $\tau'_2$. 
\begin{figure}[htb] 
	\begin{center}
	\psfrag{a}{$e$}
	\psfrag{c}{$\split_e$}
	\psfrag{b}{$e'$}
	\psfrag{d}{$c$}
	\epsfig{file=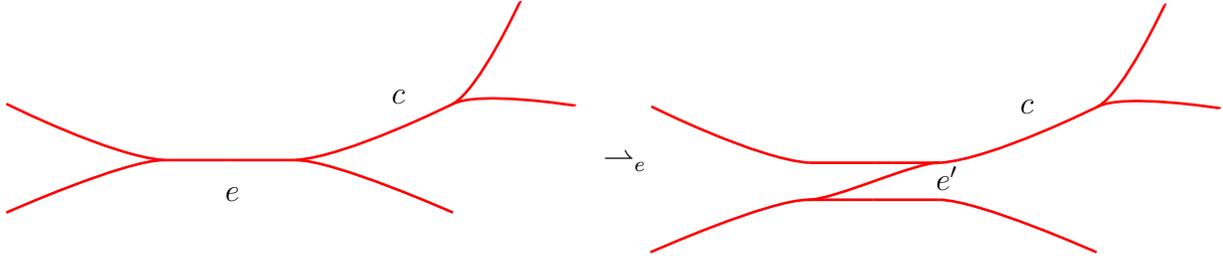, width=\textwidth}
	\caption{\label{split} Splitting a large branch to get a new large branch $c<e$.}
	\end{center}
\end{figure} 
To see the induction, we prove that if $\mu_i(e)$ is not maximal weight, then $\tau_i$ and $\tau'_{i+1}$ 
agree outside of a neighborhood of $e$ and $e'$ consisting of the edges incident with $e$. The 
point is that after splitting $e$, the only new large branches $c$ that may appear must be incident
with $e'$ (see Figure \ref{split}). However, the weight of these edges is $\mu_i(c)<\mu_i(e)$, and therefore they will not be maximal weight
in $(\tau'_{i+1}, \mu'_{i+1})$ if $e$ is not maximal weight in $(\tau_i,\mu_i)$. Thus, the maximal splittings of $\tau_i,\tau'_{i+1}$
will occur at the same edges, which will be disjoint from the edges adjacent to $e$ and $e'$, respectively. 
So we will have $\tau_{i+1}\split_e \tau'_{i+2}$. 
\end{proof}

Now, we observe that if $\L$ is suited to  $(\tau,\mu)$, then it will also be suited to $(\tau'',\mu'')$. The argument
in Corollary \ref{spliteq} can be used to show that any sequence of splits of a measured train track may be
arranged by a sequence of commuting splits to be in maximal order. 

\begin{theorem} \label{periodicsplitting}
If $\varphi: \S \to \S$ is a pA map, with stable lamination $\L^s$, and $(\tau,\mu)$ is suited to $\L^s$,
then there exists $n,m$ such that 
$$(\tau,\mu) \split^n (\tau_n,\mu_n) \split^m (\tau_{n+m},\mu_{n+m}),$$
and $\tau_{n+m} = \varphi(\tau_n)$ and $\mu_{n+m} = \lambda(\varphi)^{-1} \varphi(\mu_n)$. 
\end{theorem}
\begin{proof}
We have $\varphi(\L^s) = \lambda(\varphi) \L^s$. Therefore, $(\tau,\mu)$ and $(\varphi(\tau), \lambda(\varphi)^{-1} \varphi(\mu))$ are
equivalent measured train tracks by Theorem \ref{split1}. 

Let $$(\tau,\mu)\split (\tau_1,\mu_1) \split (\tau_2,\mu_2) \split \cdots \split (\tau_n,\mu_n) \split \cdots  $$
be the sequence of maximal splits. Then we  have 
$$(\varphi(\tau),\varphi(\mu)) \split (\varphi(\tau_1),\varphi(\mu_1)) \split \cdots \split (\varphi(\tau_n), \varphi(\mu_n))\split \cdots$$
is also a sequence of maximal splits. Since $(\tau,\mu) \sim (\varphi(\tau),\lambda(\varphi)^{-1} \varphi(\mu))$, 
by Corollary \ref{spliteq} there exists $n,m$ such that $\varphi(\tau_n)=\tau_{n+m}$ and $\varphi(\mu_n)=\lambda(\varphi) \mu_{n+m}$. 
\end{proof}

To summarize, the splitting sequence $(\tau_n,\mu_n)$ is eventually periodic, modulo the action of $\varphi$ and rescaling
the measure. One can recognize this periodicity combinatorially. See the example in Figure \ref{maximal splitting}. 

\begin{rmk} \label{excluded}
We make a remark about the kind of train tracks which can arise in such periodic splitting
sequences. There can be no edge of the train track $\tau$ which is a small branch with two branches coming
off of the same side, or a monogon (see Figure \ref{excluded2}). If $\tau$ has such an edge, then it is stable under any sequence
of folds, since the vertices involved in a fold can never involve the interior of the small branch of one of these
excluded configurations, so that the branch always remains small. So this edge will never disappear, contradicting Lemma \ref{edgesplit}. Note that the isolated
monogon is the same as a one-sided small branch when pulled back to the universal cover. 
\end{rmk}

\begin{figure}[htb] 
	\begin{center}
	\subfigure[One sided small branch]{\epsfig{figure=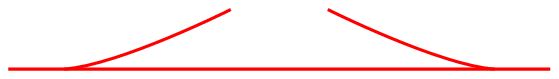,angle=0,width=.45\textwidth}}\quad
	\subfigure[Isolated monogon]{\epsfig{figure=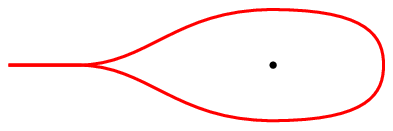,angle=0,width=.45\textwidth}}
	\caption{\label{excluded2} Excluded configurations}
	\end{center}
\end{figure}

\section{Triangulations} \label{triangulations}
An ideal triangulation $T$ of a punctured surface $\S$ is a decomposition along ideal arcs into ideal triangles. 
A {\it Whitehead move} takes any arc in $T$ which is adjacent to two distinct triangles, removes it, and replaces
it with the other diagonal of the quadrilateral (Figure \ref{Whitehead}).  
\begin{figure}[htb] 
	\begin{center}
	\epsfig{file=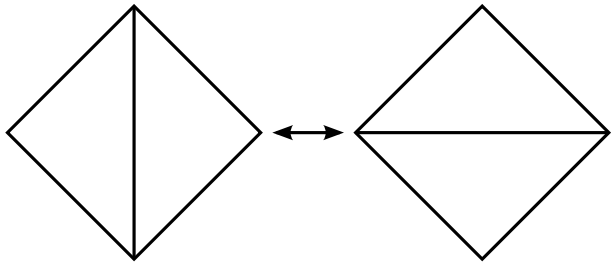, width=.5\textwidth}
	\caption{\label{Whitehead} A Whitehead move}
	\end{center}
\end{figure}

Any two ideal triangulations of $\S$ are related by Whitehead moves (see \cite[Lemma 6]{Lackenby00} for a proof). 

The {\it mapping torus} of $\varphi: \S\to \S$ is the manifold $ \T(\varphi) = \S \times [0,1] / \{ (x,0) \sim (\varphi(x),1)\}$. 
Given a pseudo-Anosov homeomorphism $\varphi:\S \to \S$, let $\S^\circ_\varphi = \S^\circ \subset \S$ be the surface obtained by
removing the singular points of the stable and unstable foliations for $\varphi$ and let $\varphi^{\circ}:\S^\circ \to \S^\circ$ denote the
restriction. 

An {\it ideal triangulation} of a  3-manifold $N$ with boundary is obtained by taking a CW complex with affinely identified
tetrahedra, such that removing the vertices gives a manifold homemorphic to the interior of $N$. A {\it taut  tetrahedron }
is a tetrahedron such that each face is assigned a coorientation, such that two faces are pointing inwards
and two are pointing outwards. Each dihedral edge of a taut tetahedron may be assigned
an angle of either $0$ or $\pi$, such that the sum of the angles around each corner of a vertex is $\pi$,
and so that each face is co-oriented in such a way that the orientations of adjacent faces change only 
along an edge of angle $0$. There is only one taut tetrahedron up to combinatorial equivalence (see Figure \ref{tauttet}(a)).
A {\it taut ideal triangulation} is an ideal triangulation of $N$ such that each triangle is assigned a coorientation
in such a way that every tetrahedron is taut, and the sum of the angles around every edge is $2\pi$ (Figure \ref{tauttet}(b)). 
One may ``pinch" the triangles together along an edge to obtain a branched
surface, so that there are cusps along the angle zero dihedral corners,
and a smooth surface between the angle $\pi$ faces of the tetrahedra.

\begin{figure}[htb] 
	\begin{center}
	\subfigure[Taut tetrahedron]{\psfig{figure=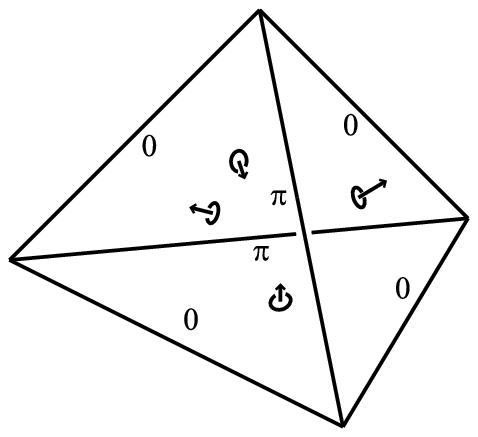,angle=0,width=.45\textwidth}}\quad
	\subfigure[Edge of taut triangulation]{\psfig{figure=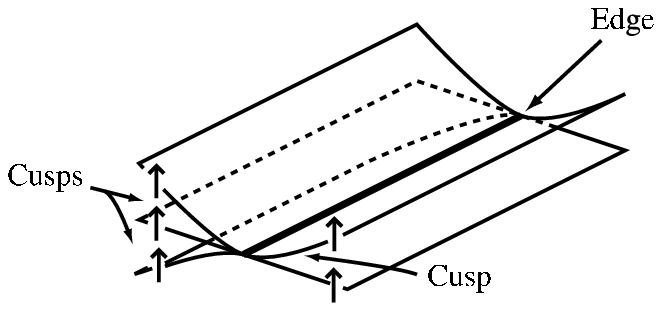,angle=0,width=.45\textwidth}}
	\caption{\label{tauttet} Conditions for a taut ideal triangulation}
	\end{center}
\end{figure} 

If we are given $\T(\varphi^{\circ})$ a mapping torus with ideal triangulation $T$ of $\S^\circ$, 
and a periodic sequence of Whitehead moves $T\to T_1 \to \cdots \to T_m = \varphi^{\circ}(T)$, we may form a taut ideal triangulation
of $\T(\varphi^{\circ})$.  Start with the triangulation $T$ of $\S^\circ$. 
We attach a tetrahedron to $\S^\circ$ along the two triangles of $T$ which are removed under
the Whitehead move to make $T_1$ (see Figure \ref{pachner}).

\begin{figure}[htb] 
	\begin{center}
	\psfig{file=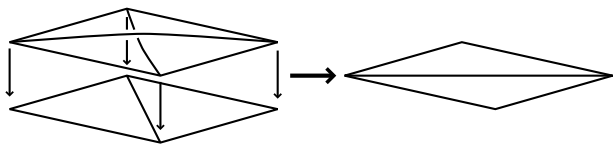, width=.5\textwidth}
	\caption{\label{pachner} Attaching a taut tetrahedron.}
	\end{center}
\end{figure} 
 The triangles are cooriented in a consistent fashion
with $\S^{\circ}$. We repeat this process until we get $T_m$, which then may be glued
to $T$ by $\varphi$. The fact that $\varphi^{\circ}$ is pseudo-Anosov guarantees that we get a
triangulation of $\T(\varphi^{\circ})$. This triangulation is called a {\it layered triangulation}.

{\bf Main Construction:}

We obtain a layered triangulation from the periodic sequence of train tracks given by Theorem \ref{periodicsplitting}.
Each train track $\tau_i$ gives a spine of $\S^{\circ}$, which is dual to a unique ideal triangulation 
$T_i$ of $\S^{\circ}$.  For each split $\tau_i \split_e \tau_{i+1}$, one has a corresponding
dual Whitehead move $T_i \to T_{i+1}$ (see Figure \ref{Whiteheaddual}). Later on, we will
also consider the reverse Whitehead move associated to a fold. 
\begin{figure}[htb] 
	\begin{center}
	\psfrag{a}{split}
	\psfrag{b}{fold}
	\epsfig{file=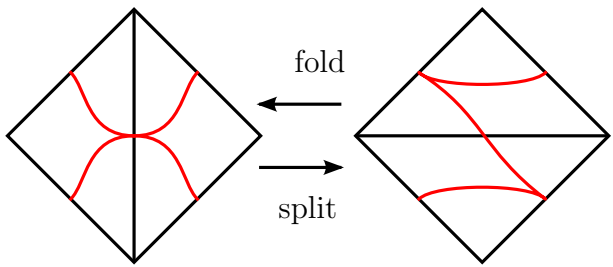, width=.5\textwidth}
	\caption{\label{Whiteheaddual} Splitting and folding}
	\end{center}
\end{figure} 
It's also clear that if we have a maximal split $\tau_i \split \tau_{i+1}$,
then this breaks up into a sequence of splits along the maximal weight branches of $\tau_i$,
and that the order in which we attach the tetrahedra does not matter since they are attached
along disjoint pairs of ideal triangles in $T_i$. 
Thus, we obtain a canonical taut ideal triangulation of $\T(\varphi^\circ)$ associated to the
mapping class $\varphi$. 

We would now like to obtain an intrinsic characterization of the taut ideal triangulations that arise
from this process.  We would like to be able to draw pictures of a taut ideal triangulation. 
In particular, we may take the preimage of the ideal triangulation in the universal cover $\tilde{\T(\varphi^\circ)}$,
and flatten it out into the universal cover of $\tilde{\S}^\circ \cong \HH^2$, in such a way
that each ideal  triangle is projected so that its coorientation agrees with that of $\tilde{\S}^\circ$. We will
draw local pictures of triangulations circumscribed by a circle corresponding to $\partial_\infty \HH^2$,
with the convention that the coorientation is pointing toward the reader.

\begin{definition}
An edge $e$ of a taut ideal triangulation is {\it left-veering} if the sequence of oriented triangles
move to the left on both sides of the edge $e$, when viewed from the edge $e$ and ordered by 
the coorientation of the triangles. Moreover, on each side of $e$ there is at least one left-veering move. Similarly, 
$e$ is right-veering if the triangles move
to the right on both sides of the edge $e$. In particular, the degree of $e$ must be at least 4 (see Figure \ref{leftveering}). 
A taut triangulation is called {\it veering} if every edge is left- or right-veering.
\end{definition}
\begin{figure}[htb] 
	\begin{center}
	\psfrag{e}{$e$}
	\subfigure[Link of a right-veering edge $e$]{\epsfig{figure=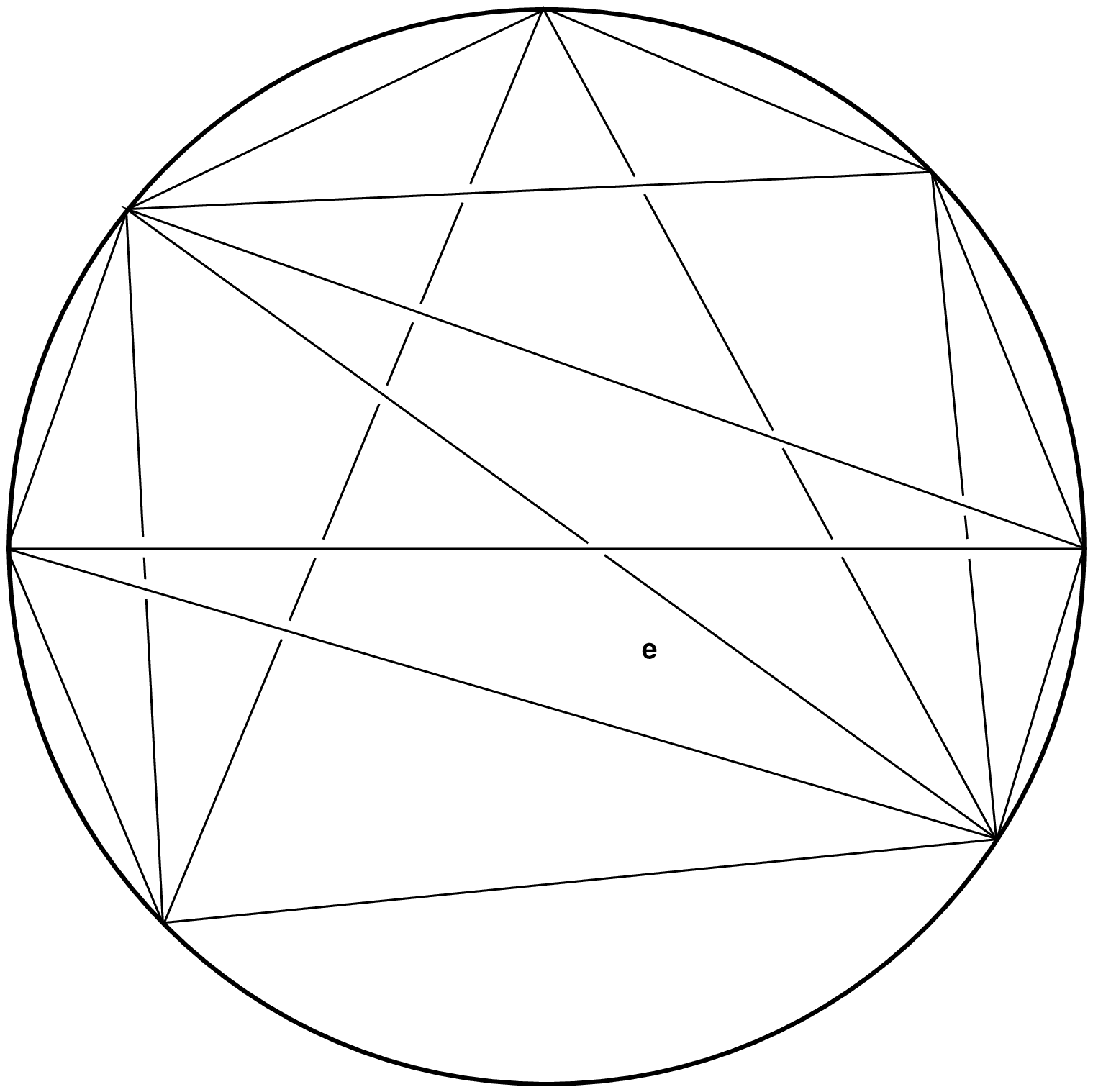,angle=0,width=.45\textwidth}}\quad
	\subfigure[Triangles adjacent to the edge $e$]{\epsfig{figure=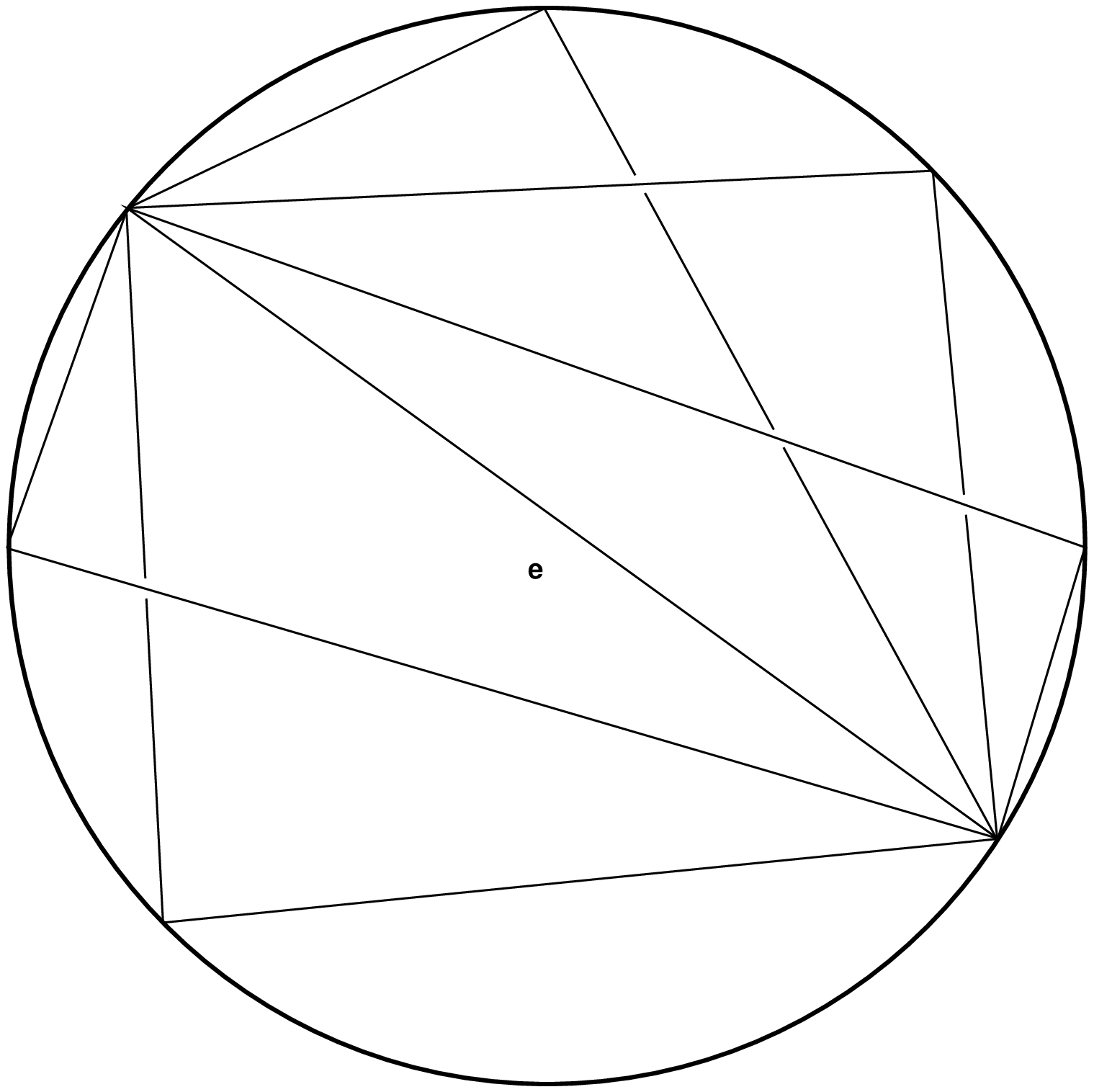,angle=0,width=.45\textwidth}}
	\caption{\label{leftveering}}
	\end{center}
\end{figure} 
{\bf Remark:} This criterion is determined purely in terms of the link of an edge in the taut ideal 
triangulation, so it makes sense even if the taut triangulation is not layered. 

\begin{prop} \label{veering}
A taut ideal triangulation of a fibered manifold coming from Whitehead moves is associated
to a periodic splitting sequence if and only if it is veering.
\end{prop}
\begin{proof}
The veering condition follows directly from the combinatorics of a splitting sequence, together
with the condition on excluded edges in Remark \ref{excluded}. In order to see these conditions,
it turns out it is necessary to work with folding sequences instead of splitting sequences.
The first time an edge $e$ appears in the sequence of Whitehead moves associated to a
folding sequence, it will be dual to a large branch. 
In order for another Whitehead
move to be applied to $e$, the dual branch must become small, and therefore foldings
must occur involving both switches of the branch, corresponding to Whitehead moves on both
sides of the edge $e$ (viewed in the universal cover $\tilde{\S}^\circ$). If one of these moves is left-veering and the other right-veering, then 
one sees that the branch dual to $e$ becomes a one-sided small branch (Figure \ref{leftright}), which is excluded
from a folding sequence (Figure \ref{excluded2}). Thus,  the initial Whitehead moves adjacent
to $e$ must be veering in the same direction. 
\begin{figure}[htb] 
	\begin{center}
	\psfrag{e}{$e$}
	\epsfig{file=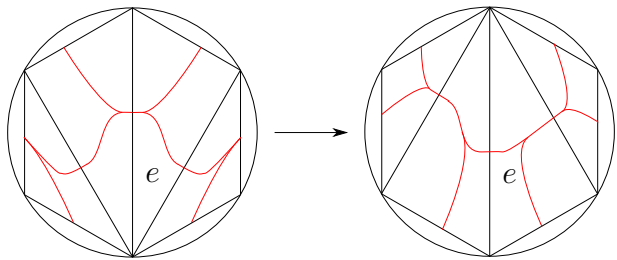, width=.5\textwidth}
	\caption{\label{leftright} Left/right Whitehead moves give an excluded configuration}
	\end{center}
\end{figure} 

The subsequent Whitehead moves on one side of $e$ must veer in the same direction as the
initial move. Initially, when the branch dual to $e$ is large, there are two edges $f, g$ in a
triangle containing $e$. When we perform a Whitehead move, say along the edge $g$, then 
the edge $g'$ replacing $g$ is dual to a large branch. Thus, any subsequent move adjacent to $e$
must be a Whitehead move dual to a fold along an edge $h$ which is not $g'$, and the edge
$h'$ replacing $h$ must be dual to a large branch after such a move (see Figure \ref{leftleft}). 
By induction, we see that the sequence of moves will always veer in the same
direction. 
\begin{figure}[htb] 
	\begin{center}
	\psfrag{e}{$e$}
	\psfrag{f}{$f$}
	\psfrag{g}{$g$}
	\psfrag{h}{$h$}
	\psfrag{i}{$g'$}
	\psfrag{j}{$h'$}
	\epsfig{file=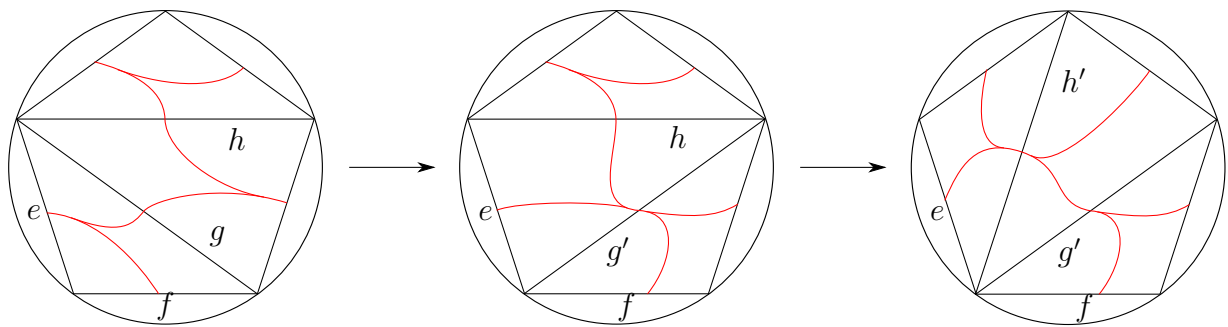, width=\textwidth}
	\caption{\label{leftleft} Left Whitehead moves must be followed by left moves}
	\end{center}
\end{figure}

Conversely, suppose we have a layered taut ideal triangulation such that the link of every edge is veering. 
We want to show how to associate to each triangle of the triangulation a train track switch. These switches
need to have the property that the induced branched surface on the boundary of a taut tetrahedron  $T$ corresponds
to a fold, with the front edge $e'$ of $T$ corresponding to the large branch, and the back edge $e$ of $T$ corresponding to the
small branch. We note that the train track on the front two faces of $T$ is uniquely determined by the fact that
the front edge $e'$ is dual to a large branch. So we choose the switch in each triangle  by the choice consistent
with the taut tetrahedron in back of the triangle. Now we check that the veering condition implies that this
choice is globally consistent. 

For each taut tetrahedron $T$, the switches in the front two faces are determined by our convention so that $e'$ is a
large branch, so we need
only check what happens on the back two faces of $T$. Take the two taut tetrahedra adjacent to these back faces. 
The veering property implies that the back edge $e$ of $T$ is a small branch of the induced train track (see Figure \ref{consistent}). Thus,
we see that for any closed loop of Whitehead moves corresponding to the layered taut veering triangulation, we
get a closed folding sequence. 
\begin{figure}[htb] 
	\begin{center}
	\psfrag{e}{$e$}
	\psfrag{f}{$e'$}
	\psfrag{T}{$T$}
		\epsfig{file=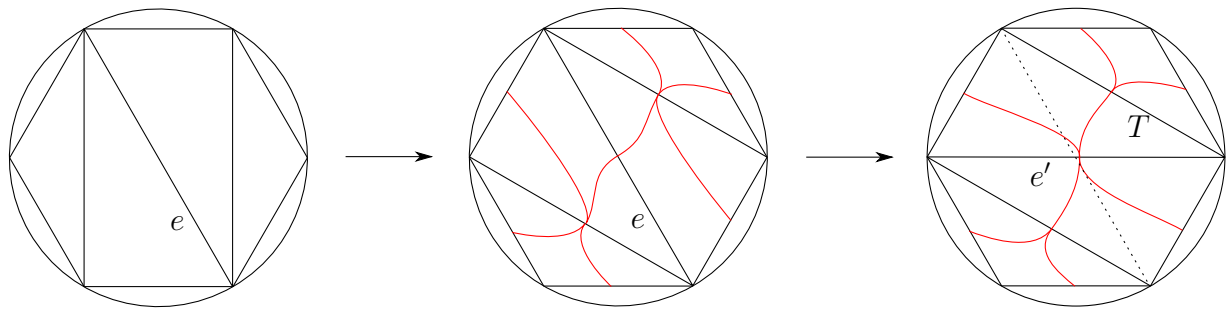, width=\textwidth}
	\caption{\label{consistent} Veering determines the train track on the boundary of $T$}
	\end{center}
\end{figure}
To finish, we need to see that the folding sequence gives rise to a maximal splitting sequence. The fact that
the folding sequence has an invariant class of projective measures follows from a Perron-Frobenius argument since the induced map of weight spaces is a Perron-Frobenius matrix, as in \cite{Penner91}. 
Thus, the invariant measure gives rise to a lamination invariant under the monodromy $\varphi$. Reversing the folding
sequence gives a splitting sequence invariant under $\varphi$. Now we apply a result of Hamenstadt which implies
that any two splitting sequences are related by commutations \cite[Section 5]{Hamenstadt09}. The proof of Hamenstadt's result may
be made similarly to the proof of Corollary \ref{spliteq}.  This implies that the layered veering triangulation is the same
as the one obtained by a maximal splitting sequence. 
\end{proof}

\begin{cor}
A periodic splitting sequence of a pseudo-Anosov mapping torus coming from $\L^s$ gives rise to a sequence
of Whitehead moves, which when reversed corresponds to a periodic splitting sequence associated
to $\L^u$. 
\end{cor}
\begin{proof}
The induced triangulation of the punctured mapping torus is veering. 
Changing the orientation of all the triangles also gives a veering triangulation. 
So the reversed sequence of Whitehead moves must also be associated to 
a periodic splitting sequence, by Proposition \ref{veering}. 
\end{proof}
This shows that the layered taut ideal triangulation associated to a pseudo-Anosov map is 
intrinsic, in that it does not depend on $\L^s$ or $\L^u$. 

\section{Example} \label{example}
For mapping classes of the punctured torus or the 4-punctured sphere, the main construction will produce
the canonical layered ideal triangulations considered in \cite{Jorgensen03, FH82}. 
For concreteness, we present the results of the main construction for the case of a 4-strand pseudo-Anosov
braid of minimal dilatation. This braid was proven to be minimal dilatation among 4-strand pseudo-Anosov
braids by Ko, Los, and Song \cite{KLS}. The associated pseudo-Anosov map $\varphi:\S_{0,5}\to\S_{0,5}$ 
has dilatation $\lambda=\lambda(\varphi)= 2.29663\ldots$, where $\lambda$ is the maximal root of the
polynomial $x^4-2x^3-2x+1$. The invariant laminations have 5 complementary regions which are punctured
monogons, and one complementary region which is a triangle. So $\S^\circ_\varphi = \S_{0,6}$, and the
mapping torus $\T(\varphi^\circ)$ is shown in Figure \ref{braid}. 
\begin{figure}[htb] 
	\begin{center}
		\epsfig{file=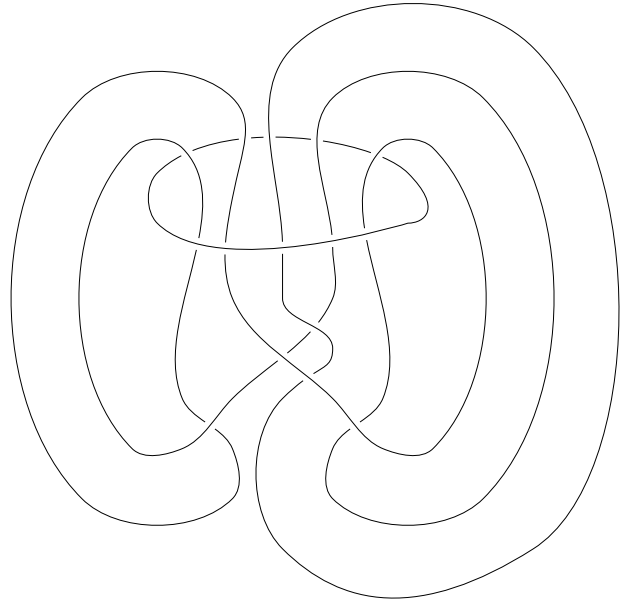, width=.5\textwidth}
	\caption{\label{braid} Braid representing $\T(\varphi^\circ)$ (the middle strand is the singular fiber)}
	\end{center}
\end{figure}
The manifold is a 3-component link complement, which has hyperbolic volume $5.33...$ (conjecturally this is the minimal
volume 3-cusped hyperbolic 3-manifold). 
Given the lamination data computed by Ko-Los-Song, we found a periodic maximal splitting sequence by hand, shown in Figure \ref{maximal splitting}.
\begin{figure}[htb] 
	\begin{center}
		\epsfig{file=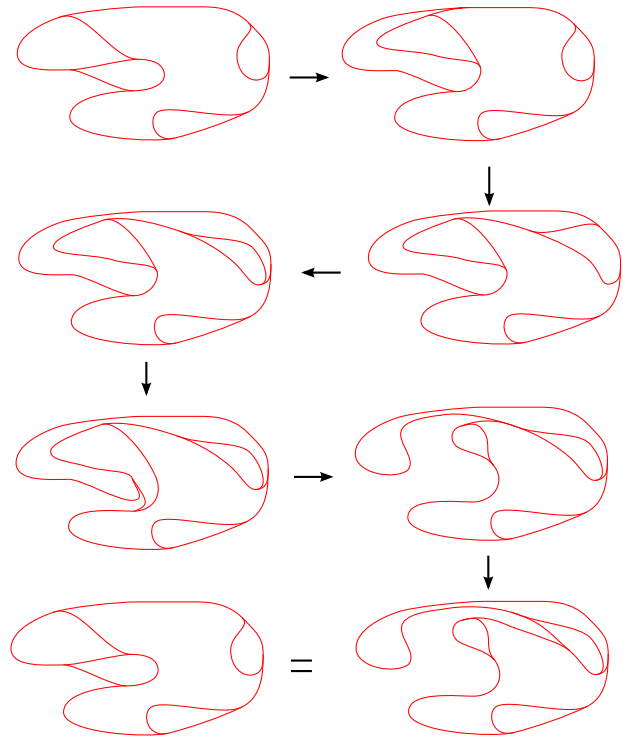, width=.5\textwidth}
	\caption{\label{maximal splitting} A maximal splitting sequence for $\varphi$}
	\end{center}
\end{figure}
Since there are six splits, the manifold has a 6 tetrahedron layered veering triangulation. 
If we represent the tetrahedra by ideal hyperbolic simplices, then they form an ideal hyperbolic tetrahedron triangulation
of the manifold, since if one of the tetrahedra had reversed orientation, the volume would be $<5.3$.

\section{Dilatation bounds} \label{dilatation}
This section gives an alternative approach to the paper \cite{FLM}. 
Given a pseudo-Anosov homeomorphism $\varphi$ of a surface $\S_{g,n}$, $\chi(\S_{g,n})<0$, let $\lambda(\varphi)$ denote its dilatation.   For
any $P > 1$, we define
\[ \Psi_P = \left\{ \varphi:\S_{g,n} \to \S_{g,n} \, \big{|} \, \varphi \mbox{ pseudo-Anosov, and } \lambda(\varphi) \leq P^{\frac{1}{2g-2+2/3n}} \right\}. \]
It follows from work of Penner \cite{Penner91} that for $P$ sufficiently large, and for each closed surface $\S_g$ of genus $g \geq 2$, there exists $\varphi_g:\S_g \to \S_g$ so that
\[ \{ \varphi_g : \S_g \to \S_g \}_{g=2}^\infty \subset \Psi_P. \]

Given a pseudo-Anosov homeomorphism $\varphi:\S \to \S$, let $\S^\circ_\varphi = \S^\circ \subset \S$ be the surface obtained by
removing the singularities of the stable and unstable foliations for $\varphi$ and let $\varphi|_{\S^\circ}:\S^\circ \to \S^\circ$ denote the
restriction. The set of pseudo-Anosov homeomorphisms
\[ \Psi_P^\circ = \left\{ \varphi|_{\S^\circ}:\S^\circ \to \S^\circ \, |\, (\varphi:\S \to \S) \in \Psi_P \right\} \]
is therefore also infinite.

\begin{lemma}
A maximal train track on $\S_{g,n}$ has at most $18g-18+6n$ branches. 
\end{lemma}
\begin{proof}
Assume every complementary region of $\tau$ is a punctured monogon or a trigon. Then
the number of monogons is at most $n$, and let $t$ denote the number of trigons. 
Let $e$ denote
the number of branches, and let $v$ denote the number of switches of the train track. 
Since each switch is in one to one correspondence with a cusp of a complementary region,
we get $v=n+3t$. Moreover, $3v=2e$ counts the number of pairs of incidences between branches and switches. 
The euler characteristic is given by $\chi(S) = 2-2g-n = v-e+ t$.
Solving for $e$, using the equations $t=(v-n)/3$ and $v=\frac23 e$ to eliminate $v, t$, we get $e= 18g-18 + 6n$.  
\end{proof}

The following theorem gives a refinement of \cite[Theorem 1.1]{FLM}. 
\begin{theorem} \label{tetrahedra bound}
Let $M \in \mathcal{T}(\Psi_P)$ be a mapping torus of a punctured
pseudo-Anosov class $M=\mathcal{T}(\varphi^\circ)$, $\varphi \in \Psi_P$. 
Then $M$ has a taut ideal triangulation with at most  $\frac12(P^9-1)$ tetrahedra. 
\end{theorem}
\begin{proof}
Let $e$ be the number of branches of a train track $\tau$ fully carrying a stable lamination for $\varphi$. By the previous lemma, $e\leq 18g-18+6n$. 
Then $\S^\circ$ has an ideal triangulation with $e$ edges. Let $\tau=\sigma_0 \to \sigma_1 \to \cdots \to \sigma_m = \varphi(\tau)$
be an invariant sequence of train track foldings coming from Theorem \ref{periodicsplitting}. The weight space $W(\sigma_i)$ has positive coordinates in a space of
dimension $e$. Folding changes the weight on the folded edge
by the sum of the weights of the two edges folded onto it (see Figure \ref{fold}). Since there is a 1-1 correspondence between the edges
of $\sigma_i$ and $\sigma_{i+1}$, we can think of the map $W(\sigma_i)\to W(\sigma_{i+1})$ as a unipotent
matrix, with the sum of entries $e+2$. Mutiplying such matrices at least adds the off-diagonal entries. 
So the induced map $W(\tau)\to W(\varphi(\tau))$ is a matrix 
with the sum of the entries at least $2m+e$. Then from \cite[Lemma 3.1]{HamSong07}, we get $2m+1 \leq \lambda(\varphi)^e $. 
By the hypothesis that $\varphi\in \Psi_P$, we have $\lambda(\varphi) \leq P^\frac{9}{e}$, so we conclude that $2m+1\leq P^9$. Thus, $\mathcal{T}(\varphi^\circ)$
has a taut ideal layered triangulation with at most $\frac12(P^9-1)$ tetrahedra. 
\end{proof}

In particular, the collection of mapping tori $\{ \mathcal{T}(\varphi^\circ) | \varphi \in \Psi_P\}$ is finite. 
Let's apply this result to the case of the closed surface of genus $g$. Let $\delta_g$ be the minimal dilatation of a pseudo-Anosov
map of a closed surface $\S_g$ of genus $g$. By a result of Hironaka and Kin, $\delta_g^{g-1}\leq 2+\sqrt{3}$.
Thus, if $\varphi$ is a pA map of $\S_g$ with dilatation $\delta_g$, then $\lambda(\varphi)=\delta_g \leq (2+\sqrt{3})^\frac{1}{g-1}= (2+\sqrt{3})^\frac{2}{-\chi(S)}$.
So we may take $P=(2+\sqrt{3})^2$ in the above theorem, and 
we obtain the number of tetrahedra of a taut ideal triangulation  is bounded by $\frac12( (2+\sqrt{3})^{18}  -1) \leq 10^{10}$. 

\section{Conclusion}
There are several natural questions arising from this paper. 

First, in the previous section we gave an upper bound on the number of tetrahedra needed for a taut ideal layered triangulation
of the punctured mapping torus of a minimal dilatation pA map of a closed surface $\S_g$. One ought to be able to improve this estimate from an improvement of the upper bound on $\delta_g$ in the papers \cite{Hironaka09, AaberDunfield10}.
Moreover, one ought to be able to improve the lower bound given in  Theorem \ref{tetrahedra bound} by getting a better estimate
of the Perron-Frobenius eigenvalues of the transformations of weight spaces. It is an interesting question to compute the numbers
$\delta_g$ or more generally $\delta_{g,n}$. The splitting sequences in this paper give a possible approach to this for a given $g,n$.
One may form finitely many automata of folding sequences for a given surface $\S_{g,n}$, one for each possible collection of 
indices of the singular points. Then one can compute the minimal dilatations for closed paths in these folding automata.  
Unfortunately, there appears to be many more train tracks which appear 
in these automata than in those considered by Ko-Los-Song  \cite{KLS, HamSong07}.

Secondly, the main construction of a taut ideal layered triangulation associated to a pseudo-Anosov mapping class gives a new
classification of conjugacy classes of pA maps. To a pA mapping class, one may associate the taut layered triangulation
of the mapping torus, together with an encoding of the homology class of the fiber. This homology class may be represented 
canonically as a harmonic simplicial 2-cycle for the taut triangulation. This amounts to assigining weights to the faces of the triangulation
such that the boundary is zero, and the signed sum of the weights on the faces of a taut tetrahedron is zero (corresponding to the dual 1-chain
being a 1-cycle). Then two pA mapping classes are conjugate if and only if there is a combinatorial equivalence between the two
triangulations which preserves the 2-cycles. This gives a new way of encoding pseudo-Anosov conjugacy classes, which seems to be somewhat
simpler than previous methods (see \cite{Hem, Mosher83, Mosher86, Mosher03}). There is also a conjugacy invariant implicit in the geometrization
theorem for mapping tori \cite{Thurstonfiber}.  In particular, our approach simplifies Hamenstadt's  conjugacy 
invariant coming from her study of cube complexes of train track splitting sequences \cite{Hamenstadt09}, but unpublished. It
would be interesting to implement this classification algorithmically, and analyze the computational complexity. 

Thirdly, it would be interesting to understand the implications of a veering taut ideal triangulation which is not layered.
There are two branched surfaces naturally associated to such a triangulation as constructed in the proof of Proposition \ref{veering}, which should carry essential laminations. 
It would be interesting to find an example of a veering triangulation which
does not come from a fibration (that is, is not layered). 

Finally, it is an interesting question whether the veering triangulations studied in Section \ref{triangulations} can be realized in the
hyperbolic metric as ideal hyperbolic triangulations with all positively oriented tetrahedra. This is true in the punctured torus case \cite{Jorgensen03, Lackenby03, Gueritaud06},
and in the example investigated in Section \ref{example}. It's possible that the approach of \cite{Gueritaud06} may be generalized to deal with this
question for veering triangulations. 

\bibliography{pAtriangulation3}

\def\cprime{$'$} \def\cprime{$'$} \def\cprime{$'$}
\providecommand{\bysame}{\leavevmode\hbox to3em{\hrulefill}\thinspace}
\providecommand{\href}[2]{#2}
\begin{thebibliography}{10}

\bibitem{AaberDunfield10}
John~William Aaber and Nathan~M. Dunfield, \emph{Closed surface bundles of
  least volume}, preprint, February 2010, \mbox{arXiv:1002.3423}.

\bibitem{FLM}
Benson Farb, Chris Leininger, and Dan Margalit, \emph{Small dilatation
  pseudo-anosovs and 3-manifolds}, 40 pages, May 2009, \mbox{arXiv:0905.0219}.

\bibitem{FH82}
W.~Floyd and A.~Hatcher, \emph{Incompressible surfaces in punctured-torus
  bundles}, Topology Appl. \textbf{13} (1982), no.~3, 263--282.

\bibitem{Gueritaud06}
Fran{\c{c}}ois Gu{\'e}ritaud, \emph{On canonical triangulations of
  once-punctured torus bundles and two-bridge link complements}, Geom. Topol.
  \textbf{10} (2006), 1239--1284, With an appendix by David Futer.

\bibitem{HamSong07}
Ji-Young Ham and Won~Taek Song, \emph{The minimum dilatation of pseudo-{A}nosov
  5-braids}, Experiment. Math. \textbf{16} (2007), no.~2, 167--179.

\bibitem{Hamenstadt09}
Ursula Hamenst{\"a}dt, \emph{Geometry of the mapping class groups. {I}.
  {B}oundary amenability}, Invent. Math. \textbf{175} (2009), no.~3, 545--609.

\bibitem{Hem}
Geoffrey Hemion, \emph{On the classification of homeomorphisms of $2$-manifolds
  and the classification of $3$-manifolds}, Acta Math. \textbf{142} (1979),
  no.~1-2, 123--155.

\bibitem{Hironaka09}
Eriko Hironaka, \emph{Small dilatation pseudo-anosov mapping classes coming
  from the simplest hyperbolic braid}, preprint, September 2009,
  \mbox{arXiv:0909.4517}.

\bibitem{Jorgensen03}
Troels J{\o}rgensen, \emph{On pairs of once-punctured tori}, Kleinian groups
  and hyperbolic 3-manifolds (Warwick, 2001), London Math. Soc. Lecture Note
  Ser., vol. 299, Cambridge Univ. Press, Cambridge, 2003, pp.~183--207.

\bibitem{Lackenby00}
Marc Lackenby, \emph{Taut ideal triangulations of 3-manifolds}, Geom. Topol.
  \textbf{4} (2000), 369--395 (electronic).

\bibitem{Lackenby03}
\bysame, \emph{The canonical decomposition of once-punctured torus bundles},
  Comment. Math. Helv. \textbf{78} (2003), no.~2, 363--384.

\bibitem{Mosher83}
Lee Mosher, \emph{Pseudo-anosovs on punctured surfaces}, Ph.D. thesis,
  Princeton, 1983.

\bibitem{Mosher86}
\bysame, \emph{The classification of pseudo-{A}nosovs}, Low-dimensional
  topology and {K}leinian groups ({C}oventry/{D}urham, 1984), London Math. Soc.
  Lecture Note Ser., vol. 112, Cambridge Univ. Press, Cambridge, 1986,
  pp.~13--75.

\bibitem{Mosher03}
\bysame, \emph{Train track expansions of measured foliations}, preprint, 2003,
  \mbox{http://andromeda.rutgers.edu/\tild@mosher/}.

\bibitem{PapadopoulosPenner87}
Athanase Papadopoulos and Robert~C. Penner, \emph{A characterization of
  pseudo-{A}nosov foliations}, Pacific J. Math. \textbf{130} (1987), no.~2,
  359--377.

\bibitem{PP87}
\bysame, \emph{A characterization of pseudo-{A}nosov foliations}, Pacific J.
  Math. \textbf{130} (1987), no.~2, 359--377.

\bibitem{PapadopoulosPenner90}
\bysame, \emph{Enumerating pseudo-{A}nosov foliations}, Pacific J. Math.
  \textbf{142} (1990), no.~1, 159--173.

\bibitem{PennerHarer92}
R.~C. Penner and J.~L. Harer, \emph{Combinatorics of train tracks}, Annals of
  Mathematics Studies, vol. 125, Princeton University Press, Princeton, NJ,
  1992.

\bibitem{Penner91}
Robert~C. Penner, \emph{Bounds on least dilatations}, Proceedings of the
  American Mathematical Society \textbf{113} (1991), no.~2, 443--450.

\bibitem{Poenaru80}
Valentin Po{\'e}naru, \emph{Travaux de {T}hurston sur les diff\'eomorphismes
  des surfaces et l'espace de {T}eichm\"uller}, S\'eminaire {B}ourbaki
  (1978/79), Lecture Notes in Math., vol. 770, Springer, Berlin, 1980, pp.~Exp.
  No. 529, pp. 66--79.

\bibitem{KLS}
Won~Taek Song, Ki~Hyoung Ko, and J{\'e}r{\^o}me~E. Los, \emph{Entropies of
  braids}, J. Knot Theory Ramifications \textbf{11} (2002), no.~4, 647--666,
  Knots 2000 Korea, Vol. 2 (Yongpyong).

\bibitem{Th}
William~P. Thurston, \emph{The geometry and topology of 3-manifolds}, Lecture
  notes from Princeton University, 1978--80,
  \mbox{http://www.msri.org/communications/books/gt3m}.

\bibitem{Th88}
William~P. Thurston, \emph{On the geometry and dynamics of diffeomorphisms of
  surfaces}, Bull. Amer. Math. Soc. (N.S.) \textbf{19} (1988), no.~2, 417--431.

\bibitem{Thurstonfiber}
William~P. Thurston, \emph{Hyperbolic structures on 3-manifolds, ii: Surface
  groups and 3-manifolds which fiber over the circle}, preprint, 1998,
  \mbox{arXiv:math/9801045}.

\end{thebibliography}

\end{document}